\documentclass[12pt,a4paper,oneside]{amsart}
\usepackage{amsmath}
\usepackage{colonequals}
\usepackage{times,fullpage}
\usepackage{amssymb}
\usepackage{amsthm}
\usepackage{geometry,microtype}
\usepackage{color}
\usepackage{graphicx}
\usepackage{hyperref}
\usepackage{enumitem}
\usepackage{dsfont}
\hypersetup{colorlinks=true}
\newtheorem{theorem}{Theorem}[section]
\newtheorem{lemma}[theorem]{Lemma}
\newtheorem{remark}[theorem]{Remark}
\newtheorem{corollary}[theorem]{Corollary}

\newtheorem{conjecture}[theorem]{Conjecture}
\newcommand{\lcm}{\operatorname{lcm}}

\author{Abhishek Jha}
\address{\parbox{\linewidth}{
Indraprastha Institute of Information Technology\\
 New Delhi, Delhi 110020, India}}
\email{abhishek20553@iiitd.ac.in}

\author{Ayan Nath}
\address{\parbox{\linewidth}{
Chennai Mathematical Institute\\
H1 Sipcot It Park, Siruseri, Kelambakkam 603103, India}}
\email{ayannath@cmi.ac.in}

\author{Emanuele Tron}
\address{\parbox{\linewidth}{Institut Fourier, Université Grenoble Alpes\\100 rue des mathématiques, 38610 Gières, France}}
\email{emanuele.tron@univ-grenoble-alpes.fr}

\title{The moments of split greatest common divisors}
\dedicatory{In memory of Prof. Florian Luca (1969--2026)}

\allowdisplaybreaks

\begin{document}

\begin{abstract}
Sequences of the form $(\gcd(u_n,v_n))_{n \in \mathbb N}$, with $(u_n)_n$, $(v_n)_n$ sums of $S$-units, have been considered by several authors. The study of $\gcd(n,u_n)$ corresponds, after Silverman, to divisibility sequences arising from the algebraic group $\mathbb G_{\mathrm{a}} \times \mathbb G_{\mathrm{m}}$; in this case, Sanna determined all asymptotic moments of the arithmetic function $\log\,\gcd (n,u_n)$ when $(u_n)_n$ is a Lucas sequence. Here, we characterize the asymptotic behavior of the moments themselves $\sum_{n \leq x}\,\gcd(n,u_n)^\lambda$, thus solving the moment problem for $\mathbb G_{\mathrm{a}} \times \mathbb G_{\mathrm{m}}$. We give both unconditional and conditional results, the latter only relying on standard conjectures in analytic number theory.
\end{abstract}

\maketitle

\section{The problem and set-up}

 GCD problems are typically concerned with the study of sequences of the shape \[(\gcd(u_n,v_n))_{n \in \mathbb N}\] (in the $2$-dimensional case), where $(u_n)_{n\in\mathbb N}$ and $(v_n)_{n \in \mathbb N}$ are sums of almost $S$-units \cite{Tr20} or divisibility sequences originating from geometry \cite{Si05}, depending on the setting; the best-known example being the sequence $\gcd(2^n-1,3^n-1)$ \cite{BCZ03} first studied by Bugeaud--Corvaja--Zannier in connection with the Subspace theorem. Such a study is usually a difficult problem that has been considered by several authors and is connected to deep topics like Vojta's conjecture in Diophantine geometry \cite{Si05}. However, for sequences of the shape $(\gcd(n,u_n))_{n \in \mathbb N}$, with $(u_n)_{n \in \mathbb N}$ a Lucas sequence, which accordingly come from orbits of points in the algebraic group $\mathbb G_{\mathrm{a}} \times \mathbb G_{\mathrm{m}}$ and its twists \cite{ST18}, the problem turns out to be more tractable with the appropriate analytic tools.

 In what follows, we let $(u_n)_{n \in \mathbb N}$ be a \emph{Lucas sequence}, that is a sequence of rational integers satisfying $u_0 = 0$, $u_1 = 1$, and $u_n = a_1 u_{n - 1} + a_2 u_{n - 2}$ for every integer $n \geq 2$, where $a_1, a_2$ are fixed nonzero relatively prime integers.
We also assume that $(u_n)_{n\in \mathbb N}$ is non-degenerate, which means that the ratio of the roots of the characteristic polynomial $X^2 - a_1 X - a_2$ is not a root of unity. In particular, the discriminant $\Delta_u := a_1^2 + 4a_2$ is nonzero.

 Given thus a fixed non-degenerate Lucas sequence $(u_n)_{n \in \mathbb N}$, the \emph{moment problem} for $\mathbb G_{\mathrm{a}} \times \mathbb G_{\mathrm{m}}$ asks to determine asymptotically moments of all orders of the sequence $(\gcd(n,u_n))_{n \in \mathbb N}$. A first step in this has been undertaken by Sanna \cite[Th. 1.1]{Sa18}, who instead considered the slower-growing sequence $(\log \gcd(n,u_n))_{n \in \mathbb N}$ and found its asymptotics, of which we state a simplified version.
\begin{theorem}[Sanna]\label{sanna}
Let $(u_n)_{n \in \mathbb N}$ be a non-degenerate Lucas sequence and $\lambda \in \mathbb N$ a positive integer. Then, as $x \rightarrow +\infty$,
\[ \sum_{n \leq x} (\log\,\gcd(n,u_n))^\lambda = C_{u,\lambda}\,x+O\left(x^{1-1/(3+3\lambda)}\right) \]
for some explicit constant $C_{u,\lambda}$ depending on the recurrence $(u_n)_{n \in \mathbb N}$ and on the integer $\lambda$.
\end{theorem}

The moments of $(\gcd(n,u_n))_{n \in \mathbb N}$ itself are significantly more difficult to determine. The only partial result in this direction is an upper bound due to Mastrostefano \cite[Th. 1.3]{Ma19}.

\begin{theorem}[Mastrostefano]
With the notation of Theorem \ref{sanna}, as $x \rightarrow +\infty$ we have
\[ \sum_{n \leq x} \gcd(n,u_n)^\lambda \leq x^{\lambda+1-(1+o(1))\sqrt{\log\log x /\log x}}   .\]
\end{theorem}

This much is known about the moment problem for $\mathbb G_{\mathrm{a}} \times \mathbb G_{\mathrm{m}}$.

{In this article, we make crucial progress on the aforementioned problem, providing precise asymptotic estimates under rather standard conjectures and strong unconditional bounds for the important case where $a_2 = \pm 1$.\footnote{ This can also be argued to be the most ``interesting'' case, see Remark \ref{a2} for details.} }

\begin{theorem}\label{maincond}

Let $(u_n)_n$ be a non-degenerate Lucas sequence  with its coefficient $a_2$ (in the characteristic polynomial) equal to $\pm 1$,\footnote{ The assumption $a_2=\pm  1$ is required only for the proof of the lower bound.} and $\lambda \in \mathbb N$ a positive integer. Assume Conjectures \ref{conjgp} and \ref{shiftedconj}. Then as $x \rightarrow +\infty$ we have
\begin{equation}\label{maincondeq}
\sum_{n \leq x} \gcd(n,u_n)^\lambda = x^{\lambda+1-(1+o_{\lambda}(1))\,\log\log\log x/\log\log x}.
\end{equation}
\end{theorem}

\medskip

Unconditionally, we obtain the following result.

\begin{theorem}\label{mainuncond}

Let $(u_n)_n$ be a non-degenerate Lucas sequence,  and $\lambda \in \mathbb N$ a positive integer. Then as $x \rightarrow +\infty$ we have the upper bound
 \begin{equation}\label{mainuncondeq}  \sum_{n \leq x} \gcd(n,u_n)^\lambda \leq x^{\lambda+1-(1/2+o_\lambda(1))\,\log\log\log x/\log\log x}. 
 \end{equation} Moreover, if the coefficient $a_2$ (in the characteristic polynomial) is equal to $\pm 1$, we  have the lower bound 
 \begin{equation}
     \sum_{n\le x}\gcd(n,u_n)^{\lambda}\ge x^{\lambda+0.715}.
 \end{equation}
\end{theorem}

It does not seem feasible to prove Theorem \ref{maincond} unconditionally. While the coefficient $1/2$ in the upper bound of Theorem \ref{mainuncond} is merely an artifact of the method, the exponent in the lower bound essentially rests on the distribution of smooth shifted primes, which is presently not known unless one is willing to assume the Elliott--Halberstam conjecture (and would equal $\lambda+1+o(1)$ in such case).

\medskip

Given Theorem \ref{maincond}, one may wonder what further progress can be made. The composite integers $n$ dividing $u_n$ generalize pseudoprimes and are called \emph{Lucas pseudoprimes}: based on this analogy, we conjecture, similarly to \cite[\S 4]{Po81} for pseudoprimes, that the moments should satisfy the following refined asymptotic.

\begin{conjecture}
With the notation of Theorem \ref{maincond}, we have as $x \rightarrow +\infty$,
\[ \sum_{n \leq x} \gcd(n,u_n)^\lambda = x^{\lambda+1} \exp \left( -\frac{\log x}{\log\log x} \left(\log\log\log x+\log\log\log\log x+o(1)\right) \right) .\]
\end{conjecture}

Our work has a number of consequences on several results and articles. We solve Conjectures 3 and 5 of \cite{Tr20} in the case $a_2=\pm 1$ and the main conjecture of \cite{LT15}; we improve \cite[Th. 1.3]{ALPS12} unconditionally thanks to a bound of Lichtman (see Corollary~\ref{Nu}); we give a new proof for the main results in \cite{LT15} and \cite{Sa17} (see Theorem~\ref{gordonupper}) and provide a matching lower bound for their conjectures; we answer a case of Question 1.1 and make progress towards Question 1.2 in \cite{ST18}; we improve Theorem 1.3 and Corollary 1.5 in \cite{Ma19}.

For instance, as an example of what could be extracted from our methods, we give the following improvement on the distribution of $\gcd(n,u_n)$ for suitably large values of $y$. 

\begin{corollary}
    We have 
    \[\#\{n \leq x:\,\gcd(n,u_n)\,>\,y\}\,\le\, x^{2-(1/2+o(1))\,\log\log\log x/\log\log x}/y\] for every $y\ge 1$, when $x$ is sufficiently large.
\end{corollary}
\begin{proof}
    Employing Theorem \ref{mainuncond} above for $\lambda=1$, we obtain 
    \[\#\{n \leq x:\,\gcd(n,u_n)\,>\,y\}\,\le\, \frac{1}{y}\sum_{n\le x} \gcd(n,u_n)\,\le\, x^{2-(1/2+o(1))\,\log\log\log x/\log\log x}/y \] for every $y\ge 1$.
\end{proof}
This estimate is an improvement of \cite[Cor. 1.5]{Ma19} and \cite[Cor. 1.3]{Sa18}, for values of $y$ larger than $x^{1-(1/2+o(1))\,\log\log\log x/\log\log x}$.

\subsection{Acknowledgements} The authors thank the anonymous referee for pointing out several issues in an earlier draft of the manuscript and many insightful comments that greatly improved the quality of the work. The first-named author would like to thank Jared Duker Lichtman for several useful discussions concerning Theorem~\ref{thm:shftdprime} of the paper.

\section{Preliminaries on Lucas sequences}

For each positive integer $m$ which is relatively prime with $a_2$, we let $z_u(m)$ be the \emph{rank of appearance} of $m$ in the Lucas sequence, that is, the smallest positive integer $n$ such that $m$ divides $u_n$.
It is well known that $z_u(m)$ exists \cite{Re13}.
Furthermore, put $\ell_u(m) := \lcm\!\big(m, z_u(m)\big)$ and, for each positive integer $n$, let $g_u(n) := \gcd(n, u_n)$.  To save space, we write $\log_k$
for the $k$-th iterate of the natural logarithm function, for all arguments for which it is defined, and define $L(x):=x^{\log_3 x/\log_2 x}$. From now on, the sequence $(u_n)_n$ will always be fixed, and all dependencies on it shall be omitted, so we will, for instance, write $z$, $\ell$, $g$, $\Delta$ instead of $z_u$, $\ell_u$, $g_u$, $\Delta_u$. We also consider $\lambda\geq 1$ to be fixed once and for all.

The following lemma collects some elementary properties of the functions $z$, $\ell$, $g$.

\begin{lemma}\label{lem:basic}
For all positive integers $m,n$ and all prime numbers $p$, with $p \nmid a_2$, we have:
\begin{enumerate}[label={\rm (\roman{*})},itemsep=0.5em]
\item \label{item1b} $m \mid u_n$ if and only if $\gcd(m, a_2) = 1$ and $z(m) \mid n$.
\item \label{item2b} $z(\lcm(m,n))=\lcm(z(m),z(n))$ whenever $\gcd(mn, a_2) = 1$.

\item \label{item3b} If $p$ is odd, then $z(p)\mid p - (-1)^{p-1}\left(\frac{\Delta}{p}\right)$, where $\left(\frac{\cdot}{p}\right)$ is the Legendre symbol. For $p=2$, $z(2)$ is $2$ if $a_1$ is even, and $3$ if $a_1$ is odd.
\item \label{item4b} $z(p^n)=p^{e(p, n)}\,z(p)$, where $e(p, n)$ is some nonnegative integer less than $n$.
\item \label{item5b} $\ell(p^n) = p^n z(p)$ if $p \nmid \Delta$, and $\ell(p^n) = p^n$ if $p \mid \Delta$.
\item \label{item6b} $g(m) \mid g(n)$ whenever $m \mid n$.
\item \label{item7b} $n \mid g(m)$ if and only if $\gcd(n, a_2) = 1$ and $\ell(n) \mid m$.
\end{enumerate}
\end{lemma}
\begin{proof}
\ref{item1b}--\ref{item4b} are well-known properties of the rank of appearance of a Lucas sequence (see e.g.~\cite{Re13}, \cite[Ch.~1]{Ri00}, or \cite[\S2]{Sa17}); while \ref{item5b}--\ref{item7b} follow easily from \ref{item1b}--\ref{item4b} and from the definitions of $\ell$ and $g$.
\end{proof}

\section{The upper bound}\label{upperbound}

To show the upper bounds in Theorems \ref{maincond} and \ref{mainuncond}, we let $n' = g(n)$ and $k=n/g(n)$, so $n=n'k$. The sum is taken over pairs $(n',k)$ such that $n'k \le x$ and $n' = \gcd(n'k, u_{n'k})$.
The condition $n' = \gcd(n'k, u_{n'k})$ implies $n' \mid u_{n'k}$. By Lemma \ref{lem:basic}\ref{item1b}, this is equivalent to the simultaneous conditions $\gcd(n',a_2)=1$ and  $z(n') \mid n'k$.
Thus, we have
\[ \sum_{n \leq x} g(n)^\lambda = \sum_{\substack{n'k \leq x \\ n' = \gcd(n'k, u_{n'k})}} (n')^\lambda \,\leq \sum_{k \leq x} \sum_{\substack{n' \leq x/k \\ z(n') \mid n'k}} (n')^\lambda\leq x^\lambda \sum_{k \leq x} \frac{1}{k^\lambda} \sum_{\substack{n' \leq x/k \\ z(n')\mid n'k}}1. \]

\noindent For ease of notation, we replace $n^{'}$ by $n$ and henceforth assume that $\gcd(n,a_2)=1$ throughout this section---for the values of $n$ that do not satisfy this, $z(n)$ is not defined, and they do not contribute to the sum anyway. We define $d_0(n,k) := \min \{d\ge 1 : d \mid k\text{ and } z(n)\mid dn \}$.
It follows that
\begin{equation}\label{d0}\sum_{k\le x} \frac 1 {k^\lambda} \sum_{\substack{n\le x/k\\ z(n) \mid nk}} 1= \sum_{k \le x} \frac 1{k^\lambda} \,\sum_{d \mid k} \sum_{\substack{n \le x/k\\ d_0(n,k)=d}} 1.\end{equation}

\noindent Note that in case $d_0(n,k)=d$ for some positive integer $n$, then $z(n)d'=dn$ for some positive integer $d'$. We infer that $\gcd(d,d')=1$ due to minimality of $d$. Keeping this observation in mind, we get 
\begin{equation}\label{eq12} \sum_{\substack{n \le x/k\\ d_0(n,k)=d}} 1 \le  \sum_{\substack{d'\le x \\ \gcd(d,d')=1}}\sum_{\substack{n \le x/k\\ z(n)/n = d/d'\\ }}1.\end{equation}

\noindent To bound this double sum, we distinguish two cases according to the size of $d'$.
\medskip

{\bf Case 1.} $d'\le x/L(x)^{13}$.

Note that since $\gcd(d,d')=1$, in order to bound the double sum on the right-hand side of \eqref{eq12} by counting admissible pairs $(d',n)$, we need to count the number of $n\le x/k$ such that $n/\gcd(n,z(n))=d'$. We can assume that $n>x/L(x)$. Since $d'\le x/L(x)^{13}$, we need to bound the number of $n$ such that \[\frac{x}{L(x) \gcd(n,z(n))}\le \frac{n}{\gcd(n,z(n))}\le \frac{x}{L(x)^{13}}.\]

Then we can see that $\gcd(n,z(n))>L(x)^{12}$ for all such $n$. This reduces the problem to bounding the number of integers $n$ for which $\gcd(n,z(n))$ is large. We establish the required bound in the Appendix (Theorem~\ref{paul}), a result which may be of independent interest. We obtain that the count of corresponding $n$'s is at most $x/L(x)^{1+o(1)}$. 

\medskip

{\bf Case 2.} $x/L(x)^{13}<d'\le x.$

Since $d'$ again must divide $n$, write $n = d'y.$
Then $z(d'y) = dy$ and $z(d')/ \gcd(z(d'),d)$ divides $y$.
Hence, for $n \le x/k$, $d'z(d')/\gcd(z(d'),d)$ must divide $n.$
Here, we have $d'z(d')\le x$, whence $z(d')\le L(x)^{13}$. \\Fix $z(d')=z$ in $\left[1,L(x)^{13}\right]$ and set ${\mathcal B}_z:=\{n \in \mathbb N: z(n)=z\}.$ Then by \cite[Th. 3]{GP91} we have the following.
\begin{lemma}[Gordon--Pomerance]\label{gp}
For $t$ large enough independently of $z$,
\begin{equation}\label{eq:B}
\# \mathcal B_z(t) \le t/L(t)^{1/2+o(1)}
\end{equation}
uniformly in $z$.
\end{lemma}

We now let $d'\in {\mathcal B}_z$. Then $n\le x/k$ is a multiple of $d'z/\gcd(z,d)$. The  number of such $n$ is at most $\lfloor x\,\gcd(z,d)/kd'z\rfloor\le x/d'z$, as $d\mid k$. Summing up the above inequality over $d'\in {\mathcal B}_z$ and using partial summation and \eqref{eq:B}, we have
\begin{align*}
\frac{x}{z}\sum_{\substack{d'\in {\mathcal B}_z \\ x/L(x)^{13}<d'\le x}} \frac{1}{d'}  &=  \frac{x}{z}\int_{x/L(x)^{13}}^x \frac{d \# {\mathcal B}_z(t)}{t} 
\\ &=  \frac{x}{z}\left(\frac{\# {\mathcal B}_z(t)}{t}\Big|_{t=x/L(x)^{13}}^{t=x} +\int_{x/L(x)^{13}}^x \frac{\# {\mathcal B}_z(t)}{t^2} d t\right)
 \\&\leq  \frac{x}{z} \left(\frac{\# {\mathcal B}_z(x)}{x}+\int_{x/L(x)^{13}} ^x \frac{d t}{t\, L(t)^{1/2+o(1)}}\right)
\\&=  \frac{x}{z}\left(\frac{1}{L(x)^{1/2+o(1)}}+\frac{1}{L(x)^{1/2+o(1)}} \int_{x/L(x)^{13}}^x \frac{dt}{t}\right)
 \\&=   \frac{x}{z\, L(x)^{1/2+o(1)}},
\end{align*}

where in the above calculation we used the fact that 
\[
L(t)^{1/2+o(1)}=L(x)^{1/2+o(1)}\quad {\text{\rm uniformly~in}}\quad t\in \left[x/L(x)^{13},x\right]\quad {\text{\rm as}}\quad x\to+\infty.
\]

We now sum over $z\in \left[1,L(x)^{13}\right]$ and obtain that
\[ \sum_{\substack{x/L(x)^{13}\le d'\le x \\ (d,d')=1}}\sum_{\substack{n \le x/k\\ z(n)/n = d/d',\\ }}1\le \frac{x}{L(x)^{1/2+o(1)}}\sum_{1\le z\le L(x)^{13}} \frac{1}{z}\\ =\frac{x}{L(x)^{1/2+o(1)}}.\]

\medskip

Combining Cases 1 and 2, we have  \[\sum_{\substack{n \le x/k\\ d_0(n,k)=d}}1\le\frac{x}{L(x)^{1/2+o(1)}}. \] 

Substituting this upper bound in (\ref{d0}),  \[\sum_{k \leq x} \frac{1}{k^\lambda}\, \sum_{d \mid k}\sum_{\substack{n \le x/k\\ d_0(n,k)=d}}1  \le \sum_{k \leq x} \frac{\tau(k)}{k^\lambda}\,\cdot\frac{x}{L(x)^{1/2+o(1)}} \le \frac{x}{L(x)^{1/2+o(1)}} \] where the last inequality follows as the sum $\sum_{k \leq x} {\tau(k)}/{k^\lambda}$ is $O((\log x)^2)$ if $\lambda=1$ and finite otherwise.
This proves the unconditional upper bound in Theorem \ref{mainuncond}.

\medskip

For the conditional upper bound, we recall the following standard conjecture.

\begin{conjecture}\label{conjgp}
With the same hypotheses as Lemma \ref{gp}, we have
\begin{equation}\label{pomeranceconj}\# \mathcal B_z(t) \leq t/L(t)^{1+o(1)}.\end{equation}
\end{conjecture}

This conjecture can be justified on several grounds. For instance, a heuristic proof is given in \cite[Th. 8.10]{Ze14}, while in \cite[\S 2]{Po80} the same bound is proved with the totient function in place of $z$; this conjecture originates from the same conjecture in the setting of pseudoprimes \cite{Po81}, where an identical bound is presumed to hold.

Assuming Conjecture \ref{conjgp}, in order to prove that the moment sum of the GCDs is at most the right-hand side in \eqref{maincondeq}, one just has to run the whole proof in this section while replacing all instances of $t/L(t)^{1/2+o(1)}$ with $t/L(t)^{1+o(1)}$. This concludes the proof of the upper bounds in both Theorems \ref{maincond} and \ref{mainuncond}. 

\medskip 

Let \[\mathcal{N}_u(x):=\{n\in \mathbb N\cap [1,x]\,:\, \gcd(n,a_2)=1\text{ and }g(n)=n\,\}.\] Our argument also provides a new analytic proof of the main theorem in \cite[Th. 1.2]{Sa17}.
\begin{theorem}\label{gordonupper}
Let $(u_n)_n$ be a non-degenerate Lucas sequence. As $x \to +\infty$, one has
\[ \#\mathcal{N}_{u}(x)\le x/L(x)^{1/2+o(1)}. \]
If we moreover assume Conjecture \ref{conjgp}, we also have $\#\mathcal{N}_{u}(x)\le x/L(x)^{1+o(1)}$.
\end{theorem}
\begin{proof}
Note that \[\#\mathcal{N}_u(x)=\sum_{\substack{d\le x }}\sum_{\substack{n \le x\\ z(n)/n = 1/d\\ (n,a_2)=1  }}1.\] This sum is a special case of (\ref{eq12}), so the same argument applies.
\end{proof}

\section{The conditional lower bound}
Our strategy for the conditional lower bound is to show that the main contribution to the sum comes from a specific subset of integers. We will focus on the set $\mathcal{N}_{u}(x)$, which consists of numbers $n \leq x$ that divide $u_n$. We will first demonstrate, via partial summation, that if this set is sufficiently large, it provides the required lower bound. The remainder of the proof is then dedicated to establishing a sufficiently large cardinality for $\mathcal{N}_{u}(x)$ by constructing a large subset $\mathcal{L} \subseteq \mathcal{N}_{u}$ composed of numbers with specific smoothness properties related to shifted primes.

For $X, Y \ge 1$, we let $P(n)$ be the greatest prime factor of the integer $n$, and set
\[ \Pi(X, Y) = \#\{p\le X: P(p-1) \le Y\}. \]
It is natural to expect that shifted primes are smooth with the same relative frequency as generic numbers of the same size. This heuristic, coupled with the prime number theorem, suggests that
\begin{equation}\label{eq:pomconjecture} \Pi(X,Y) \sim \frac{\Psi(X,Y)}{\log{X}} \end{equation}
in a wide range of $X$ and $Y$\!. An explicit conjecture of this kind appears in \cite{Po80}.

\begin{conjecture}[Pomerance]\label{pomconj}
Estimate \eqref{eq:pomconjecture} holds whenever $X,Y\to\infty$ with $X\ge Y$.
\end{conjecture}
This conjecture has been used to solve various problems on arithmetic functions \cite{Po19}. 

\medskip

In our context, involving a Lucas sequence with discriminant $\Delta$, the relevant analogue concerns primes $p$ where $p-(\Delta/p)$ is smooth, rather than $p-1$. Therefore, we define:
\[\Pi^{*}(X, Y) = \#\{p\le X: P(p-(\Delta/p)) \le Y\}. \] Based on the ideas of Granville, it was proved in \cite[Lemma 6.2.1]{Ag22} that, under the weak Elliott--Halberstam conjecture, it holds that
\[ \Pi^{*}(X,Y) \sim \frac{\Psi(X,Y)}{\log{X}} \] where $U=\log X/\log Y$ is a fixed constant greater than $1$. In the same vein, it is reasonable to assume, just like in (\ref{eq:pomconjecture}), the following weak instance of Granville's conjecture to hold in a larger range of $U$.
\begin{conjecture}\label{shiftedconj} 
Suppose that $X,Y\to\infty$ with $X\ge Y$ and that $U=\log X/\log Y  =(1+o(1))\sqrt{\log X}$. Then we have
\begin{equation}\label{eq:komconjecture} \Pi^{*}(X,Y) \gg \frac{\Psi(X,Y)}{\log{X}}. \end{equation}
\end{conjecture}
\noindent This assumption plays a crucial role in our proof.

\medskip

We will show that the set $\mathcal N_u(x)$ is already large enough to contribute the main term to the sum.

To show the conditional lower bound on moments of $g(n)$, we argue as follows: by partial summation,
\[
\sum_{n\le x}\,g(n)^{\lambda} \geq \sum_{\substack{ n \le x \\ n \in \mathcal{N}_u}} n^\lambda = x^\lambda \#\mathcal{N}_u(x) - \int_1^x \lambda t^{\lambda-1} \#\mathcal{N}_u(t) \, dt. \]

We now invoke some facts from calculus \cite[Prop. 1.5.8]{BGT87} to prove a nontrivial integral estimate.

\begin{lemma}[Karamata's integral theorem]\label{karamata}
Suppose that a measurable locally bounded function $\ell : \left [e,+\infty\right) \to \mathbb R^+ $ is such that, for each fixed $c>0$, one has \[ \lim_{x \rightarrow +\infty} \frac{\ell(cx)}{\ell(x)}=1\] (such an $\ell$ is called a \emph{slowly varying} function in the sense of Karamata). Let moreover $\alpha>-1$. Then as $x \to +\infty$,
\[ \int_e^x t^\alpha\, \ell(t)\,dt \sim \frac{x^{\alpha+1}}{\alpha+1}\, \ell(x). \]
\end{lemma}

\begin{lemma} \label{calculus}
Let $f:\left [1, +\infty \right ) \rightarrow \left [ 1, +\infty \right )$ be non-decreasing, such that $1 \leq  f(x) \leq x$ for all $x \geq 1$ and moreover $f(x) = x/L(x)^{1+o(1)}$ as $x \to +\infty$. Then for any fixed $\lambda \geq 1$, as $x \to +\infty$
\[x^\lambda f(x) - \lambda \int_1^x t^{\lambda-1} f(t) \, dt \ge x^{\lambda+1} /L(x)^{1+o(1)}. \]
\end{lemma}
\begin{proof}

For any fixed $\varepsilon>0$, we have that \[  t/L(t)^{1+\varepsilon} \le f(t) \le t/L(t)^{1-\varepsilon}  \] for $t$ large enough. Since we are interested in asymptotics as $x \to +\infty$, we may as well suppose that this holds for $t \ge e$. The function $t \mapsto L(t)$ is slowly varying by direct calculation: indeed, note that $\log_2(cx)=\log_2 x+\log(1+\log c/\log x) = \log_2 x+O(1/\log x)$, and likewise $\log_3(cx) =\log_3 x+O(1/\log x \log_2 x)$, so that
\[
\log \frac{L(cx)}{L(x)} =\frac{\log(cx) \log_3(cx)}{\log_2(cx)}-\frac{\log x \log_3 x}{\log_2 x} = \log c\, \frac{\log_3 x}{\log_2 x} + O\left ( \frac{1}{(\log_2 x)^2} \right ) \rightarrow 0.\]

The function $t \mapsto L(t)^{-(1-\varepsilon)}$ is then slowly varying as well \cite[Prop. 1.3.6(ii)]{BGT87} and we may apply Lemma \ref{karamata} to it. 

Now, fix any $e < y \le x$. From the above discussion, Karamata's theorem, and the condition that $f$ is non-decreasing, we find
\[ \int_e^x  t^{\lambda-1} f(t) \, dt \leq \int_e^y \frac{t^\lambda}{L(t)^{1-\varepsilon}} \,dt + \int_y^x t^{\lambda-1} f(x)\, dt = \frac{1+o(1)}{\lambda+1}\frac{y^{\lambda+1}}{L(y)^{1-\varepsilon}}+f(x)\,\frac{x^\lambda-y^\lambda}{\lambda}
.\]

We next make the choice $y=x\,L(x)^{-3\varepsilon}$ to obtain that
\begin{align*} x^\lambda f(x) - \lambda \int_e^x t^{\lambda-1} f(t) \, dt &\ge  f(x) y^\lambda-\frac{(1+o(1))\lambda}{\lambda+1}\frac{y^{\lambda+1}}{L(y)^{1-\varepsilon}}
\\ &\ge \frac{x}{L(x)^{1+\varepsilon}} \left ( \frac{x}{L(x)^{3\varepsilon}}\right )^\lambda- \frac{(1+o(1))\lambda}{\lambda+1}\frac{(x/L(x)^{3\varepsilon})^{\lambda+1}}{L(x)^{1-\varepsilon+o(1)}}
\\ & = x^{\lambda+1} \left ( \frac{1}{L(x)^{1+\varepsilon+3\lambda\varepsilon}} - \frac{1}{L(x)^{1-\varepsilon+o(1)+3\varepsilon(\lambda+1)}}   \right )
\end{align*} 
(since $L(y)=L(x)^{1-o(1)}$). The exponent $1+\varepsilon+3\lambda\varepsilon$ is less than the exponent $1-\varepsilon+o(1)+3\varepsilon(\lambda+1)$ for $x$ large enough, and as a consequence \[ x^\lambda f(x) - \lambda \int_e^x t^{\lambda-1} f(t) \, dt \ge \frac{x^{\lambda+1}}{L(x)^{1+(3\lambda+1)\varepsilon}}. \]
Since $\varepsilon$ can be taken to be arbitrarily small, we recover the claimed estimate for $x \to +\infty$.

\end{proof}

Assume for the moment that \[ \#\mathcal N_{u}(t)\ge \frac{t}{L(t)^{1+o(1)}}  \] for large enough values of $t$, so that $\#\mathcal N_{u}(t)= t/L(t)^{1+o(1)} $ as well by  Conjecture~\ref{conjgp}. Then we may apply Lemma \ref{calculus} with $f(t)=\#\mathcal N_{u}(t)$ to conclude that 
\[ \sum_{n\le x}\,g(n)^{\lambda} 
\ge x^\lambda \#\mathcal{N}_u(x) - \int_1^x \lambda t^{\lambda-1} \#\mathcal{N}_u(t) \, dt
 \ge  \frac{x^{\lambda+1}}{L(x)^{1+o(1)}}, \]
which is the sought lower bound, i.e. the lower bound half of Theorem \ref{maincond}.

\bigskip

We now adapt the proof of \cite[Th. 1.3]{ALPS12} to show our claim that $\#\mathcal{N}_u(x)$ is large. For any real number $y\ge1$ let
\[
M_y\colonequals \lfloor y\log^{3} y\rfloor !.
\]
We construct a set $\mathcal{L}$ of integers whose cardinality serves as a lower bound for the one of $\mathcal{N}_u$, with the following rule: a positive integer $n$ is in $ \mathcal{L}$ if it is of the form \[n = 2sM_y\] for some $y\ge3$ and for some odd
squarefree positive integer $s$ such that: 
\begin{itemize}
\item $\gcd(s,M_ya_{2})\allowbreak=1$, and 
\item for every prime $p\mid s$ we have $p-(\Delta\, /\, p)\mid M_y$. 
\end{itemize}

We now show that $\mathcal{L}\subseteq\mathcal{N}_u$
for any Lucas sequence $(u_n)_{n \in \mathbb N}$ with $a_2=\pm 1$. Such a restriction ensures that $z(m)$ is defined for all positive integers $m$.
It suffices to show: for any
$n=2sM_y\in\mathcal{L}$ and for any prime power
$q\mid n$, we have $z(q) \mid n$.
This is easy for $q\mid s$, since then $q=p$ is an odd prime and either
$z(p)=p$ (in case $p\mid\Delta$) or
$z(p)\mid p-(\Delta\, /\, p)$ (otherwise).  And since $p-(\Delta\, /\, p)\mid M_y$, we have
$z(p)\mid n$ in either case.

If on the other hand $q \mid 2M_y$, we have two cases depending on the parity of $q$.
\begin{itemize}
    \item When $q$ is odd, we have that $q=p^{k}$ with $p$ prime such that $p\le y(\log y)^{3}$. We have that $z(q)\mid (p-1)p^{k-1}, \,p^{k},\text{ or }(p+1)p^{k-1}$. If $p+1\le \lfloor y\log^{3} y\rfloor$, then $z(q)\mid M_y$ as $p^{k-1}\mid p^{k}\mid M_y$. The remaining case is when $p+1>\lfloor y\log^3 y\rfloor $ and $k=1$ and $z(q)=p+1$. Writing $p+1=2^{j}m$ where $m$ is odd, we can see that $2^j\mid 2M_y$ and $m\mid M_y$. Thus, in all cases, $z(q)\mid 2M_y$.
    \item When $q=2^{k}$, we know that $z(q)\mid 2^{k}\text{ or }3\cdot2^{k-1}$. Since $y\ge 3$, we get that $z(q)\mid 2M_y$.
\end{itemize}

\medskip

We now use the method of Erdős~\cite{Er35} to show that the set 
$\mathcal{L}(x):=\mathcal{L}\cap[1,x]$ is rather large. For this we take 
$$
y := \frac{\log x}{(\log_{2} x)^{6}} \text{ and } z :=\exp((\log_{2}x)^{2}-\log_{2}x\log_{4}x).
$$
Let $v\ge 2$.
Define $\mathcal{P}$ as the set of primes  $p$ such that:
\begin{itemize}
\item  $p \in [y(\log y)^3, z]$;
\item  $p-(\Delta/p)$ is $y$-smooth.
\end{itemize}

Thus, by (\ref{eq:komconjecture}), we obtain
$$
\#\mathcal{P} \geq \Pi^{*}(z,y) +O( y\log^3 y)  \gg \frac{\Psi(z,y)}{\log z}+O( y(\log y)^3).
$$
Now applying standard estimates for $\Psi(X,Y)$ \cite[Th. 2.1]{Po89}, we further find that, with $U:=\log z/\log y$,
\[\#\mathcal{P} \gg \frac{\Psi(z,y)}{\log z} \gg \frac{z\exp(-(1+o(1))U\log U)}{\log z}=z\exp(-(1+o(1))\log_{2}x\log_{3}x).\]

Next, we show that for any squarefree positive integer $s$ 
composed out of primes $p \in \mathcal{P}$, the integer $n = 2sM_y$ belongs to $ \mathcal{L}$. Clearly, $\gcd(s,M_ya_2)=1$. Let $p\mid s$. Then we need to show that $p-\left(\Delta/p\right)\mid M_y$. Let $q^{k}$ be a prime power dividing $p-\left(\Delta/p\right)$. Then $2 \le q\le y$ and \[\nu_{q}(p-\left(\Delta/p\right))\le \log (z+1)/\log q\le (\log_2 x)^
{2}/ \log q< 2(\log_{2} x)^2.\] At the same time, \[\nu_q(M_y)\ge \frac{y\,(\log y)^3}{2q}\ge \frac{1}{2} (\log_2 x-6\log_3 x)^3\ge \nu_{q}(p-\left(\Delta/p\right)) ,\] for large enough $x$. This proves our claim.

Consider the set $\mathcal{L}_r(x)$ of all such integers $n = 2sM_y$,
where  $s$ is composed out of 
$$
r :=\left\lfloor{\frac{\log x}{(\log_{2} x)^{2}}}\right\rfloor
$$ 
distinct primes $p\in \mathcal{P}$. 
By the trivial estimate $n!\le n^{n}$, we have 
$M_y \le \exp(\log x/(\log_{2} x)^2)$ for large $y$; moreover $s \le z^r$. This implies that  
for sufficiently large $x$ we have $n \le x/L(x)^{(1+o(1))\log_{4}x/\log_{3}x}$ for 
every $n \in \mathcal{L}_r(x)$. 

As for the cardinality of $\mathcal{L}_r(x)$, note that $(\#\mathcal{P}-i)/(r-i) \ge \#\mathcal{P}/r$ for $1 \leq i <r \leq \#\mathcal{P}$ and multiply this over $1 \leq i \leq r-1$ to get
$$\# \mathcal{L}_r(x) \ge \binom{\#\mathcal{P}}{r} \ge  \left(\frac{\#\mathcal{P}}{r}\right)^r. $$
Using the estimate for $\#\mathcal{P}$ and $r$:
\[ \frac{\#\mathcal{P}}{r} \gg \frac{z \exp(-(1+o(1))\log_2 x \log_3 x)}{(\log x)/(\log_2 x)^2} \]
\[ = \frac{\exp((\log_2 x)^2 - (1+o(1))\log_2 x \log_3 x)}{(\log x)/(\log_2 x)^2}. \]
This leads to the bound $\#\mathcal{L}_r(x)\ge  {x}/{L(x)^{1+o(1)}}$.
 Noting that $\mathcal{L}_r(x)\subseteq\mathcal{L}(x)\subseteq \mathcal{N}_u(x)$, our proof is complete.

\begin{remark}
    One can reproduce the arguments in \cite[\S 5]{Wr20} to get a smooth set of shifted primes as in $\mathcal{P}$ above under the Heath-Brown conjecture on the least prime in an arithmetic progression. However such an argument would only lead us to the bound \[\#\mathcal{N}_u(x)\gg \frac{x}{L(x)^{2+o(1)}}.\] 
\end{remark}

\section{The unconditional lower bound}
In this section, we obtain an unconditional lower bound for the case where $a_2=\pm 1$. The proof begins similarly to the previous section and the only difference is in the process  we employ  to pick a  smooth set of shifted primes. Our choice of these sets of primes is based on a slight modification of a recent result of Lichtman on smooth shifted primes.
\begin{theorem}[Smooth shifted primes in APs with fixed moduli]\label{thm:shftdprime}
For fixed nonzero $a\in\mathbb{Z}$ and $\beta > 15/32\sqrt{e}= 0.284\ldots$, there exists $C\ge1$ such that
\begin{align}
\sum_{\substack{x<p\le 2x\\P(p-1) \le x^\beta \\ a\mid p-1}}1 \ \gg \ \frac{x}{(\log x)^C}.
\end{align}
\end{theorem}
\begin{proof}
The proof only requires a slight modification in the construction of suitable moduli as in  \cite[\S 3]{Li22}, whose notation we now borrow, and the rest of the argument remains the same.
Let $\theta=17/32-\varepsilon$ and take $\beta>(1-\theta)/\sqrt{e}$. Define the integer $H = \lceil (2\theta-1)/\varepsilon\rceil$; in particular $\frac{2\theta-1}{H}\in[\varepsilon/2,\varepsilon]$. Define the subset $L = \{l\sim x^{\frac{2\theta-1}{H}}\}$ used in dyadic decomposition, and define the sequence $\mathcal G = \{al_1\cdots l_H : l_i\in L\}$, so that $l\asymp x^{2\theta-1}$ for all $l\in \mathcal G$, and
\[
|\mathcal G|  = |L|^H \gg x^{2\theta-1}
\]

Consider
\begin{align*}
\mathcal N \ & = \ \big\{(p,l,m,n) \;:\; p-1 = lmn, \ l\in\mathcal G, \, m,n\sim x^{1-\theta},\,\, p\sim x \big\}\\
\mathcal N' \ & = \ \big\{(p,l,m,n)\in \mathcal N \;:\; P^+(p-1) \le x^\beta \big\}
\end{align*}

Note that we have incorporated a constant factor $a$ in the definition of $\mathcal{G}$. This ensures that the primes considered in $\mathcal{N}$ satisfy the congruence $p\equiv 1\bmod{a}$ without altering the properties of $\mathcal{G}$ as defined in \cite[\S 3]{Li22}.
By proceeding exactly as in \cite[\S 3]{Li22}, we now obtain that
\[\label{eq:N1bound}
\#\mathcal N' \gg \frac{x}{\log x}.
\]
and then deduce the stated bound. \end{proof}
Once we have this theorem, we can run the same argument as in the beginning of the previous section, while exploiting the theorem that we just proved. There are two main differences: the choice of primes dividing $s$ and the quantity $M_y$.

Let first $v\in \left(1, 3.517\right]$. We set 
\[ y := \frac{\log x}{\log_{2} x} \text{ and } z := y^v,\]
\[M_y:= \mathrm{lcm}\left\{m: m\le y\right\}.\]

Let us define $\mathcal{P}$ as the set of primes  $p$ such that:
\begin{itemize}
\item  $p \in \left(z, 2z\right)$;
\item  $p\equiv 1\pmod{8|\Delta|}$ so that $(\Delta/p)=1$;
\item  $p-1$ is $y$-smooth;
\item  $p-1$ is not divisible by any 
proper prime power $q > y$. 
\end{itemize}

For any proper prime power $q$, there are at most $O(z/q)$ primes $p \le 2z$ such that $q\mid p-1$. Moreover, it is easy to see that there are only  $O(
t^{1/2})$ proper prime powers $q\le t$. By partial summation, 
it follows that the number of primes not satisfying the last condition is at most $O(z/y^{1/2})$.

Thus, by Theorem~ \ref{thm:shftdprime} applied with $a=8|\Delta|$ and $x=z$, we obtain
$$
\#\mathcal{P} \gg \frac{z}{(\log z)^{C}} + O(zy^{-1/2}) \ge  \frac{C_0z}{(\log z)^{C}},
$$
for some constants $C \ge 1$ and $C_0>0$, provided that $x\ge x_0$ for a sufficiently large constant $x_0$. 

As in the previous section, consider the set $\mathcal{L}_v(x)$ of all integers \[n = 2sM_y,\]
where  $s$ is composed out of 
$$
r :=\left\lfloor{\frac{\log x}{v\log_{2} x}}\right\rfloor
$$ 
distinct primes $p\in \mathcal{P}$. 
Since by the prime number theorem the estimate 
$M_y =\exp((1+o(1)) y)$ holds as $x\to+\infty$, and moreover $ s < (2z)^r$,
we see that for sufficiently large $x$ we have $n \le x$ for 
every $n \in \mathcal{L}_v(x)$. 

For the cardinality of $\mathcal{L}_v(x)$ we have
$$\# \mathcal{L}_v(x) \ge \binom{\#\mathcal{P}}{r} \ge  \left(\frac{\#\mathcal{P}}{r}\right)^r . $$

Since 
\[
r \ge \frac{\log x}{v\log_{2} x}-1
\quad\text{ and }\quad
\frac{\#\mathcal{P}}{r} \ge \frac{C_1(\log x)^{v-1}}{(\log_{2} x)^{C+v-1}}
\]
for some constant $C_1>0$, we obtain \[\# \mathcal{L}_v(x) \ge \frac{x^{1-1/v}}{L(x)^{1+(C-1)/v+o(1)}}\ge x^{1-1/v}/L(x)^{C+1},\] for large enough $x$. 

On the other hand, for each element $n=2sM_y$ in $\mathcal{L}_v(x)$, we have $M_y=\exp((1+o(1))y)$ as well as \[s \ge z^r  \ge \left( {y^v} \right )^{\log x/v\log_2 x-1}= x\, L(x)^{-1+o(1)},\]
which taken together imply that $n \ge x\, L(x)^{-1+o(1)}$. 

Now, we have that $\mathcal{L}_v(x)\subseteq \mathcal{N}_{u}(x)$. The proof is the same as in the previous section, except that we need to show that if a prime power $q\mid 2M
_y$, then $z(q)\mid 2M_y$: but the proof of this is identical to the one in \cite[p. 283]{ALPS12}.

To show the unconditional lower bound on moments of $g(n)$, we then argue as follows:
\[\sum_{n\le x}\,g(n)^{\lambda} \geq \sum_{\substack{ n\le x 
\\ n\in \mathcal{L}_v(x)}}\, n^{\lambda} \ge \left( \frac{x}{L(x)^{1+o(1)}}  \right)^\lambda \, \frac{x^{1-1/v}}{L(x)^{C+1}} = x^{\lambda+1-1/v+o(1)},\] as each element in $\mathcal{L}_{v}(x)$ is at least $x\, L(x)^{-1+o(1)}$. By plugging $v=3.517$ in the argument above, our proof of Theorem \ref{mainuncond} is complete, since $1-1/3.517+o(1)>0.715$ for $x \to +\infty$.

\medskip

As a byproduct, we also find the following unconditional estimate, which improves on \cite[Th. 1.3]{ALPS12}

\begin{corollary}\label{Nu}
  As $x \to +\infty$, we have \[\#\mathcal{N}_u(x)\gg x^{0.715}.\]
\end{corollary}

\medskip

\begin{remark}\label{a2}
One could consider the distribution of $\mathcal{N}_{u}$ in the case $a_{2}\neq \pm 1$. We believe that in this case the count $\#\mathcal{N}_u(x)\ll x^{1-\varepsilon}$ holds for some $\varepsilon>0$. \\For an integer $n$ in $\mathcal{N}_u$, the orbit $\{z(n),z(z(n)), \ldots\}$ of a discrete dynamical system is well defined. In particular, $a_{2}\nmid z^{\circ i}(n)$ for any $i\ge 0$. If we consider the simplest case when $u_{n}=c^{n}-1$ for $c\geq 2$, then $z(n)$ is the multiplicative order of $c$ modulo $n$. By a result of Ford \cite[Cor. 1]{Fo14}, we know that for some prime $r\ge 3$ the cardinality of set of positive integers less than $x$ such that $r\nmid \varphi^{\circ k}(n)$, for all iterates of the Euler totient function, is bounded by $x^s$ for some constant $s<1$. Based on his arguments, one can see that a similar result for iterates of the Carmichael lambda function $\lambda(n)$, which behaves similarly as $z(n)$, holds. We would expect similar result to be true for iterates of $z(n)$ as well, and thus the expected bound on $\mathcal{N}_u$.
\end{remark}

\section{Appendix: A bound for large GCDs}
In this appendix, we will prove an analogue of a bound of Pollack for $\gcd(n,z(n))$; the proof follows the general strategy of \cite[Th. 1.1]{Po11}. Such investigations of greatest common divisors involving integers and their valuations at arithmetic functions are an active area of research \cite{Yu22,Po22}. Our goal is to obtain explicit estimates, but we do not need the best ones.

For any real number $x\ge 1$, put 
\[G(x) \,:=\, \#\{n \le x : \,\gcd(n,a_{2})=1\text{ and } \gcd(n,z(n))>L(x)^{12}\}.\]
\begin{theorem}\label{paul}
    We have that $G(x) \le x/L(x)^{1+o(1)}$.
\end{theorem}

We first need a variant of \cite[Lemma 2.2]{Po11} for our purposes.
\begin{lemma}\label{lemma 2.2}
If $m > m_0$ is squarefree, then there is some divisor $d$ of $m$ with $\gcd(d, z(d)) \allowbreak = 1$ and $d \ge C_1m^{1/2}$ for some constant $C_1>0$.
\end{lemma}
\begin{proof}
By replacing $m$ by $m/\gcd(m,2\Delta a_2)$ if necessary, we assume that $\gcd(m,2\Delta a_2)\allowbreak =1$.
Let $m=p_1\,p_2\cdots p_k$ such that $p_k$ is the largest prime divisor of $m$.
Then we define $m_1=q_1\,q_2\cdots q_r$ where $q_1=p_k$ and $q_i$ is the largest prime divisor of $m$ which does not divide $z(q_1)\,z(q_2)\cdots z(q_{i-1})$ from the set of prime factors of $m$ not yet included in $m_1$. By construction, $q_1 > q_2 > \dots > q_r$.
The process of constructing $m_1$ stops when any remaining prime factor $p'$ of $m/m_1$ satisfies $p' \mid z(m_1)$. If there were a prime $p''$ in $m/m_1$ such that $p'' \nmid z(m_1)$, then $p''$ would have been included as some $q_{r+1}$ in $m_1$. Let $m_2 = m/m_1$. Thus, $m_2$ has the property that all its prime factors divide $z(m_1)$. Since $m_2$ is squarefree (as $m$ is), this implies $m_2 \mid z(m_1)$.
We have $\gcd(m_1,z(m_1))=1$: by construction, no $q_i$ divides $z(q_j)$ for $j<i$. If $q_i \mid z(q_j)$ for $i<j$, then $q_j>q_i$ which is a contradiction as per our construction. 
Recalling the upper bound $z(n) \leq 2n$ \cite{Sa75}, we find that
 \[2 m_1 \geq z(m_1) \geq m_2,\quad\text{whence}\quad 2 m_1^2 \geq m_1 m_2 = m.\]
Hence $m_1 \ge m^{1/2}/\sqrt{2}$. So if we choose $d = m_1$ and $C_1=1/2\sqrt{|\Delta a_2|}$,
the lemma is proved.
\end{proof}

The next lemma is an easy consequence of the Brun–Titchmarsh inequality; for a proof
see, e.g., \cite[Lemma 6]{KAct}.

\begin{lemma}
    Let $m$ be a positive integer. For all $x\ge 1$, we have \[\sum_{\substack{ p\le x \\ p\equiv \pm 1\bmod{m}}}\frac1{p}\ll \frac{\log_2 x}{\varphi(m)}.\] Here the implied constant is absolute.
\end{lemma}
Next, we need a variant of \cite[Lemma 7]{Po192} (also see \cite[Lemma 2.4]{Po11}). The proof runs exactly as in \cite{Po11} but we include it here for completeness.

\begin{lemma}\label{lemma 2.4}
    Let $d$ be a squarefree integer. Then the number of squarefree $n\le x$ such that $\gcd(n,a_2)=1$ and $d\mid z(n)$ is at most \[\omega(d)^{\omega(d)} \,\frac{x}{d} (C (\log_{2}x)^2)^{\omega(d)}\] for some constant $C>0$.
\end{lemma}
\begin{proof}
    Let \[\varphi^{*}(n)=\prod_{p\mid n} \left(p - (-1)^{p-1}\eta(p)\right)\] be a modified Euler function, where $\eta(p)=(\Delta/p)$ when $p\neq 2$ and $\eta(2)= \mathds{1}_{2\nmid \Delta}$. Clearly, $z(n) \mid \varphi^{*}(n)$ for all squarefree $n$. 
    
Since $d$ divides $\varphi^{*}(n)$, we can write $d = \prod_{p \mid n} a_p$, where each $a_p$ divides $p - (-1)^{p-1}\eta(p)$. Discarding those $a_p = 1$, we see that $n$ induces a (not necessarily unique) factorization of $d$, whereby by a \emph{factorization} of $d$ we mean a decomposition of $d$ as a product of factors strictly larger than $1$, where the order of the factors is not taken into account. For each possible factorization of $d$, we estimate the number of $n \leq x$ as in the lemma statement which induce this factorization.

Let $d = a_1 a_2 \cdots a_k$ be a factorization of $d$. If $n$ induces this factorization, then there are distinct primes $p_1, \ldots, p_k$ dividing $n$ with $p_i \equiv -1\text{ or }0\text{ or }1 \pmod{a_i}$ for each $1 \leq i \leq k$. So, the number of such $n \leq x$ is bounded above by 
\[x\prod_{i=1}^{k}\left(\sum_{\substack{p_i\le x\\ p_i\equiv -1,0,1\bmod{a_i}}}\frac1{p_i}\right)\]
For each $i$, the inner sum is bounded by $\ll \frac{\log_{2} x}{\varphi(a_i)}+\frac1{a_i}\ll \frac{(\log_{2} x)^{2}}{a_i}$. Thus, for an absolute constant $C>0$, the number of $n$ is at most 
$$
x \, \prod_{i=1}^k \frac{C (\log_2 x)^{2}}{a_i} 
= \frac{x}{d} (C (\log_2 x)^{2})^k 
\leq \frac{x}{d} (C (\log_2 x)^{2})^{\omega(d)}.$$ (The last inequality uses the observation that each factorization of $d$ involves at most $\omega(d)$
factors.) Since $d$ is squarefree, the number of factorizations of $d$ is given by $B_{\omega(d)}$
, where $B_l$
(the $l^{\text{th}}$ Bell number ) stands for the number of set-partitions of an $l$-element set.
Since any partition of an $l$-element set involves at most $l$ components, we always
have $B_l\le  l^{l}$. Taking $l = \omega(d)$ completes the proof.
\end{proof}

Let us now prove our theorem.
\begin{proof}
First, we prove a squarefree variant of Theorem \ref{paul}. Let \[G^{*}(x) := \#\{n \le x\text{ squarefree } :\, \gcd(n,a_{2})=1,\, \gcd(n,z(n))>L(x)^{8}\}.\]

We now show that that $G^{*}(x) \le x/L(x)^{1+o(1)}$. 
For each $n$ contributing to $G^\ast(x)$, if there is a prime $p>L(x)^{4}$ dividing $\gcd(n,z(n))$, then we may write $n=pr$ with $p\mid z(r)$. By Lemma \ref{lemma 2.4}, the number of possible values of $r\le x/p$ is at most $Cx (\log_{2} x)^{2} /p^{2}$. Summing over all primes $p>L(x)^{4}$, we get that the number of such $n$ is at most \[\frac{x (\log_{2} x)^{2}}{L(x)^{4}}.\]

We may therefore assume that the largest prime dividing $\gcd(n,z(n))$ is at most $L(x)^{4}$. Since $\gcd(n,z(n)) > L(x)^{8}$,
successively stripping primes from $\gcd(n,z(n))$, we must eventually discover a divisor $D_0$ of $n$ such that $D_0 \in \left(L(x)^{4}, L(x)^{8}\right]$. If $x$ (and hence $L(x)$) is large, we can apply Lemma \ref{lemma 2.2} to this divisor $D_0$ to obtain a divisor $d$ of $D_0$ (hence $d$ is also a divisor of $\gcd(n,z(n))$) with $d \in \left(C_1L(x)^{2}, L(x)^{8}\right]$ having the property that $\gcd(d, z(d)) = 1$.

Now observe that if $n=de$, then $\gcd(d,e)=1$ since $n$ is squarefree. Thus, $z(de)\mid z(d) z(e)$. As a result, we get that $d\mid z(e)$. By Lemma \ref{lemma 2.4}, the number of possible $e$ is at most \begin{equation}\label{eq}
 \quad \frac{x}{d^{2}} (C\omega(d) \log_{2}^{2} x)^{\omega(d)}.
\end{equation}

The remainder of the proof is a two-way strategy. Let $A(x)=\log x \log_{3} x/(\log_2 x)^{2}$. First, suppose that $\omega(d)<A(x)$. Then we can bound (\ref{eq}) by $xL(x)^{1+o(1)}/d^{2}$. Summing over these values of $d\in \left(C_1L(x)^{2}, L(x)^{8}\right)$, we get that the number of corresponding $n$ is at most \[x\,L(x)^{1+o(1)}\sum_{d>C_1L(x)^{2}}\frac1{d^{2}}\le \frac{x}{L(x)^{1+o(1)}}.\] The remaining $n$ have a divisor $d\in (C_1L(x)^{2},L(x)^{8}]$ with $\omega(d)>A(x)$, so that $\omega(n)>A(x)$. By a theorem of Hardy--Ramanujan \cite{HR17}, the number of $n\le x$ with $\omega(n)>A(x)$ is at most \[\frac{x}{\log x}\, \sum_{k>A(x)} \frac{(\log_2 x +O(1))^{k-1}}{(k-1)!}\le \frac{x}{L(x)^{1+o(1)}}.\] Combining the above estimates, we get that $G^{*}(x)\le x/L(x)^{1+o(1)}$.

\medskip

To transition from $G^{*}$ to $G$, let $n$ be any integer such that $\gcd(n,a_2)=1$ and $\gcd(n,z(n))>L(x)^{12}$, and write $n=n_{0}n_1$,where $n_0$ is squarefree, $n_1$ is squarefull, and $\gcd(n_0,n_1)=1$. Then, for $x$ large,
\begin{align*}
    L(x)^{12} < \gcd(n,z(n)) &\le \gcd(n_0,z(n_0))\,\gcd(n_0,z(n_1))\,\gcd(n_1,z(n_0)z(n_1)) \\&\le \gcd(n_0,z(n_0))\,z(n_1)\,n_{1}\le 2\gcd(n_0,z(n_0)) n_1^{2}.
\end{align*}
Thus, either (a) $n_1 > \frac{L(x)^{2}}{2}$ or (b) $n_1 < \frac{L(x)^{2}}{2}$ and $\gcd(n_0,z(n_0))>2L(x)^{8}$. The number of $n\le x$ in (a) is $O(x/L(x))$, while the number of $n\le x$ in (b) is at most \[\sum_{\substack{ n_1\le L(x)^{2} \\ n_{1}\text{ squarefull}}} G^{*}(x/n_1).\]
By our bound on $G^{*}(x)$ and noting that $L(x)^{1+o(1)}=L(x/n_1)$ for $n_1\le L(x)^{2}$, we get an upper bound \[\frac{x}{L(x)^{1+o(1)}}\sum_{n_{1}}\frac1{n_1}\le \frac{x}{L(x)^{1+o(1)}},\] where we used the fact that the sum of reciprocals of squarefull integers converges.
\end{proof}

\end{document}